\documentclass[11pt]{amsart}

\usepackage{amscd}
\usepackage{xypic}  
\usepackage{amssymb}
\usepackage{amsthm}
\usepackage{epsfig}
\usepackage[normalem]{ulem}
\usepackage{tikz}
\usetikzlibrary{arrows}
\usetikzlibrary{arrows.meta,arrows}
\usetikzlibrary{arrows.meta, positioning}
\usetikzlibrary{arrows,positioning,shapes,fit,calc}

\usepackage{url}

\usepackage{tikz}

\usepackage[all,cmtip]{xy}
\usepackage{amsmath}
\usepackage{color}

\newtheorem{thm}{Theorem}[section]  

\newtheorem{lemma}[thm]{Lemma}
\newtheorem{prop}[thm]{Proposition}

\newtheorem{corollary}[thm]{Corollary}

\theoremstyle{definition}

\newtheorem{remark}[thm]{Remark}

\newtheorem{question}[thm]{Question}

 \newtheorem{defn}[thm]{Definition}

 \newtheorem{example}[thm]{Example}

\def\min{\operatorname{min}}

\def\c1{\operatorname{c_1}}
\def\c2{\operatorname{c_2}}

\def\Sym{\operatorname{Sym}}

\def\mult{\operatorname{mult}}

\def\ZZ{{\mathbb Z}}

\def\PP{{\mathbb P}}

\def\DD{{\mathbb D}}

\def\L{{\mathcal L}}

\def\N{{\mathcal N}}
\def\O{{\mathcal O}}

\def\E{{\mathcal E}}

\def\H{{\mathcal H}}
\def\HH{{\mathfrak H}}

\def\Y{{\mathcal Y}}

\def\XX{{\mathfrak X}}

\def\e{\mathfrak{e}}

\def\f{\mathfrak{f}}
\def\c{\mathfrak{c}}

\def\cong{\simeq}

\def\leq{\leqslant}
\def\geq{\geqslant}

\def\+{\oplus}               
\def\*{\otimes}                  

\def\Pic{\operatorname{Pic}}

\def\Num{\operatorname{Num}}

\begin{document}

\title{Rational curves and Seshadri constants on Enriques surfaces}

\author[C.~Galati]{Concettina Galati}
\address{Concettina Galati, Dipartimento di Matematica e Informatica, Universit\`a della \linebreak Calabria, via P. Bucci, cubo 31B, 87036 Arcavacata di Rende (CS), Italy}
\email{concettina.galati@unical.it}

\author[A.~L.~Knutsen]{Andreas Leopold Knutsen}
\address{Andreas Leopold Knutsen, Department of Mathematics, University of Bergen, Postboks 7800,
5020 Bergen, Norway}
\email{andreas.knutsen@math.uib.no}

\date{July 4, 2023}

\begin{abstract}
  We prove that classes of rational curves on very general Enriques surfaces are always $2$-divisible. As a consequence, we compute the Seshadri constant of  any big and nef line bundle on a very general Enriques surface, proving that it coincides with the value of the $\phi$-function introduced by Cossec.
\end{abstract}

\maketitle

\section{Introduction}

It is well-known that the only curves with negative self-intersection on Enriques surfaces are $(-2)$-curves, which are smooth and rational, and that
a {\it general}\footnote{By the word ``general'' in this note we mean that there is a nonempty Zariski-open subset in the moduli space satisfying the given conditions.}  Enriques surface does not contain such curves (cf. references in \cite[p. 577]{cos2}).
At the same time, an Enriques surface contains many elliptic pencils, and each such contains in general $12$ singular (nodal) rational fibers. Since any elliptic pencil has two double fibers, those rational curves are (numerically) $2$-divisible on the surface. Until recently, nothing has been known concerning the existence of rational curves of higher arithmetic genera.
However, a very recent result of Baltes \cite[Thm. 1.1]{Bal} says that any elliptically fibered $K3$ surface contains a sequence of irreducible rational curves $R_i$ with $R_i^2 \to \infty$. Since any Enriques surface admits a double {\'e}tale cover by such a $K3$ surface, one may draw the consequence that any Enriques surface also contains a sequence of rational curves whose self-intersections tend to $\infty$. One can however not deduce any information on precisely which kinds of linear systems do contain rational curves, even on a general Enriques surface. 

One of the two main results in this note is that there are strong constraints on linear systems containing rational curves on {\it very general}\footnote{By the phrase ``very general'' in this note we mean that there is a set that is the complement of a countable union of proper Zariski-closed subsets in the moduli space satisfying the given conditions.} Enriques surfaces, namely that they should be (numerically) $2$-divisible:

\begin{thm} \label{mainthm1}
  Let $S$ be a  very general Enriques surface. If $C \subset S$ is an irreducible rational curve, then $C$ is $2$-divisible in $\Num(S)$.
  \end{thm}

This result was already known under the additional assumption that the rational curves be {\it nodal} (and with a completely different proof, see \cite[Prop. 1]{indam}). The proof strategy in this note is to degenerate $S$ to a union of an elliptic ruled surface and a blown-up plane, using the recent results on flat limits of Enriques surfaces from \cite{kn-JMPA,cdgk}, which will be recalled in \S \ref{sec:fl}, and prove that rational curves do not exist in any of the components of the limit Hilbert scheme when the curve classes are non-$2$-divisible. The proof will be given in \S \ref{sec:RC}. 

The following remains an open question:

\begin{question} \label{q}
  Does every (numerically) $2$-divisible complete linear system on a very general Enriques surface contain irreducible rational curves? If not, what are the conditions to contain such curves?
    \end{question}

By contrast, it has been recently proved in \cite[Thm. 1.1]{cdgk} that every non-$2$-divisible  complete linear system of genus $p \geq 2$ on a very general Enriques surface contains irreducible (nodal) curves of any geometric genus $g$ with $1 \leq g \leq p$.

The second main result of this note is obtained as an
application of  Theorem \ref{mainthm1} and concerns {\it Seshadri constants}. These were introduced by Demailly in \cite{de} to capture the concept of local positivity of line bundles.
For any point $x$ on a smooth projective variety $X$ and $H$ any big and nef line bundle
on $X$, the real number
\[\varepsilon(H,x):=\inf_{C \ni x}\frac{C \cdot H}{\mult_xC}\]
 is the \emph{Seshadri constant of $H$ at $x$}, where the
   infimum is taken over all curves $C$ passing through
   $x$ (it is easily seen that one can restrict to considering only irreducible curves).
The \emph{(global) Seshadri constant of $H$} is defined as 
\[ \varepsilon(H):=\inf_{x \in X} \varepsilon(H,x). \]
One has $\varepsilon(H) >0$ if and only if $H$ is ample, by Seshadri's criterion (see, e.g., \cite[Thm. 1.1]{prim}), and
$\varepsilon(H,x) \leq \sqrt[n]{H^n}$, where $n=\dim X$, by Kleiman's theorem
(see, e.g., \cite[Prop. 2.1.1]{prim}.

One of the very subtle points about Seshadri constants is that their values are only known in few cases. For instance, all known examples are rational and it is not even known whether  Seshadri constants are always rational or not. 
We refer to \cite{ba-sur,PAG,prim}, to mention a few, for accounts on Seshadri constants.

On Enriques surfaces, some results on Seshadri constants were obtained by Szemberg in \cite[\S 3]{Sz}, most notably, the fact that $\varepsilon(H)$ is always rational and that there always is a curve $C$ with a point of multiplicity $m$ such that $\varepsilon(H)=\frac{C \cdot H}{m}$ (see \cite[Thm. 3.3]{Sz} and its proof). Another simple observation relates the Seshadri constant to the $\phi$-invariant defined by Cossec \cite{cos2}. For any effective line bundle or divisor $H$ on an Enriques surface with $H^2>0$ one sets
\[ \phi(H):=\min\{ E \cdot H \; | \; E \; \; \mbox{is effective, nontrivial and  } \; E^2=0\},\]
which is a number that has many interesting geometrical implications (see, e.g., \cite{cd,KLpn,kn-man,Sz}) and has played a fundamental role in the study of Enriques surfaces. In particular, it is known that $\phi(H)^2 \leq H^2$ \cite[Cor, 2,7,1]{cd}.
 Clearly one has $\varepsilon(H) \leq \phi(H)$. We prove that, quite surprisingly, equality holds on very general Enriques surfaces:

\begin{thm} \label{mainthm2}
  Let $S$ be a very general Enriques surface and $H \in \Pic(S)$ big and nef.
  Then $\varepsilon(H)=\phi(H)$.
\end{thm}

The proof is given in \S \ref{sec:ses}: we first prove a result of independent interest (Proposition \ref{prop:sessub}) holding on {\it any} Enriques surface, namely that if $C$ is an irreducible curve with a point of multiplicity $m$ such that $\frac{C.H}{m} <\phi(H)$, then $C$ is rational and $C^2=m(m-1)-2$, which
in particular enables us to deduce a couple of lower bounds on Seshadri constants in Corollaries \ref{cor:mainthm2_2} and \ref{cor:mainthm2_1}.
Moreover, if $C^2>0$, then $m=\phi(C)+1$ and $C$ belongs to a short list of particular divisors;  in this list, only one case is $2$-divisible, with $m=7$, and is therefore the only case that can occur on a very general Enriques surface by Theorem \ref{mainthm1}. Using the same degeneration as above, we however show that a point of multiplicity as high as $7$  cannot occur on any limit curve.

\vspace{0.3cm} 
\noindent
{\it Acknowledgements.} C.~Galati acknowledges funding from  the GNSAGA of INdAM, the ERASMUS+ Staff Mobility Programme   and the Trond Mohn 
 Foundation Project ``Pure Mathematics in Norway''. A.~L.~Knutsen acknowledges funding from the Meltzer Foundation.

\section{Background results} \label{sec:fl}

We recall that a divisor $E$ on an Enriques surface is called {\it isotropic}
if $E^2=0$.

\subsection{Moduli spaces of polarized Enriques surfaces} 

We recall results from \cite{kn-JMPA}.

For any effective line bundle $L$ on an Enriques surface $S$ we set
\[
\varepsilon_L= \begin{cases}
  0, & \mbox{if $L+K_S$ is not $2$-divisible in $\Pic S$,} \\
  1, & \mbox{if $L+K_S$ is $2$-divisible in $\Pic S$}.
\end{cases}
\]

\begin{prop} \label{prop:exdec}
  Let $L$ be an effective line bundle on an Enriques surface $S$ such that $L^2 >0$. Then there are unique nonnegative integers $a_i$ (depending on $L$)  satisfying
  \[ a_1 \geq \cdots  \geq a_7 \;\;\mbox{and} \;\; a_9+a_{10} \geq a_0 \geq a_9 \geq a_{10}\]
such that there 
exist effective isotropic divisors 
  $E_1,\ldots,E_{10}$ (depending on $L$) such that $E_i \cdot E_j=1$ for all $i \neq j$, and an effective isotropic divisor $E_{9,10} \sim \frac{1}{3}\left(E_1+\cdots+E_{10}\right)-E_9-E_{10}$ so that we may write
\begin{equation} \label{eq:fundpres}
  L \sim a_1E_1+\cdots +a_7E_7+a_9E_9+a_{10}E_{10}+a_0E_{9,10}+\varepsilon_LK_S.
  \end{equation}
\end{prop}

\begin{proof}
  This is \cite[Prop. 2.3 and 5.5]{kn-JMPA}. 
\end{proof}

We note that the divisors $E_i$ and $E_{9,10}$ are not unique (cf. \cite[Rem. 5.6]{kn-JMPA}).

Following \cite[Def. 5.8]{kn-JMPA} we call the coefficients $a_i=a_i(L)$, $i \in \{0,1,\ldots,7,9,10\}$ and $\varepsilon_L$ the {\it fundamental coefficients}
  of $L$ or of $(S,L)$, and \eqref{eq:fundpres} a {\it fundamental representation} of $L$ or of $(S,L)$.

  Then \cite[Thm. 5.9]{kn-JMPA} states that:

  \begin{thm} \label{thm:mainmod}
    The irreducible components of the moduli space of polarized Enriques surfaces are precisely the loci parametrizing pairs $(S,L)$ with the same fundamental coefficients.
  \end{thm}

 \subsection{Flat limits of Enriques surfaces} 

We recall results from \cite[\S 3]{kn-JMPA} and \cite[\S 2]{cdgk} and refer to these papers for all proofs.  

Let $E$ be a smooth elliptic curve, $R:=\Sym^2(E)$ and
$\pi: R \to E$ be the (Albanese) map sending $x+y$ to $x\+ y$, where
$\+$ denotes the group operation on $E$.  
We set $\f_e:=\pi^{-1}(e)$ for any $e \in E$ and denote the algebraic equivalence class of the fibers by $\f$. 

Let $\eta$ be any of the three nonzero $2$-torsion points of $E$ and define
\[ T:= \{ e+ (e\+\eta) \; | \; e \in E\}. \]
Note that $\pi|_T:T \to E$ is an {\'e}tale double cover.
Embed $T$ as a cubic in $\PP^2$. Consider nine general points $y_1,\ldots,y_9 \in T$ such that
\begin{equation}
  \label{eq:cond}
  y_1+\cdots+y_9 \in |\N_{T/R} \* \N_{T/\PP^2}|,
\end{equation}
and denote by $P$ the blowing up of $\PP^2$ along $y_1,\ldots,y_9$. We denote the strict transform of $T$ on $P$ by the same name and the exceptional divisor on $P$ over $y_i$ by $\e_i$.
Consider the surface $X:=R \cup_T P$ obtained by gluing
$R$ and $P$ along $T$. 

A Cartier divisor, or a line bundle, $\L \in \Pic X$, is a pair
$(L',L'')$ such that $L' \in \Pic R$, $L'' \in \Pic P$ and $L'|_T \cong
L''|_T$. We denote by $\xi$ the Cartier divisor on $X$ represented by the pair $\xi= (T,-T)$ (cf. \cite[(3.3)]{fri2}).

We will make use of the following:

\begin{thm} \label{thm:deform} There is a flat family $f:\mathfrak{X} \to \DD$ over the unit disc such that $\mathfrak{X}$ is smooth and, setting $\mathfrak{X}_t=f^{-1}(t)$, such that
    \begin{itemize}
    \item $\mathfrak{X}_0=X$, and
     \item $\mathfrak{X}_t$ is a smooth general Enriques surface. 
    \end{itemize}

Furthermore, there is a short exact sequence
\begin{equation} \label{eq:piclimite}
  \xymatrix{
      0 \ar[r] & \ZZ \cdot \xi \ar[r] & \Pic X \cong H^2(\mathfrak{X},\ZZ) \ar[r]^{\iota_t^*} & H^2(\mathfrak{X}_t,\ZZ) \cong \Pic \mathfrak{X}_t  \ar[r] & 0,}
  \end{equation}
where $\iota_t: \mathfrak{X}_t \subset \mathfrak{X}$ is the inclusion. 
\end{thm}

\begin{proof}
  This is \cite[Thm. 2.2]{cdgk}, or \cite[Prop. 3.7, Thm. 3.10 and Cor. 3.11]{kn-JMPA}.
\end{proof}

The latter result both guarantees that $X$ smooths to a general Enriques surface and that all Cartier divisors on $X$ lift in the deformation. This means that for any line bundle $L_0=(L_R,L_P)$ on $X$, there is a line bundle $\L$ on $\XX$ such that
$\L|_X=L_0$.

We will in the following be particularly interested in the lifting of the following isotropic divisors on $X$:

\begin{example} \label{ex:iso}
Let $|A|$ be a pencil on $P$ satisfying $A^2=0$ and $T \cdot A=-K_P \cdot A=2$. Then its general member is a smooth rational curve. The induced $g^1_2$ on $T$ has, by Riemann-Hurwitz, two members that   are also fibers of $\pi|_T:T \to E$.
In other words, there are two fibers $\f_{\alpha(A)}$ and $\f_{\alpha'(A)}$ of $\pi:R \to E$ such that
$\f_{\alpha(A)} \cap T$ and $\f_{\alpha'(A)}  \cap T$ belong to this $g^1_2$. 
One easily verifies that $\alpha'(A)=\alpha(A) \+\eta$. In particular, there are two uniquely defined points $\alpha(A)$ and $\alpha(A) \+ \eta$ on $E$ such that the pairs
\begin{equation} \label{eq:excar}
  (\f_{\alpha(A)},A) \; \; \mbox{and} \; \; (\f_{\alpha(A)\+\eta},A)
\end{equation}
define isotropic Cartier divisors on $X$, whose numerical class we simply denote by $(\f,A)$. Their linear systems each contain one member.
One may check that their difference is $K_{X}$ and that
$(\f,A) \cdot (\f,A')=A \cdot A'$ for any two such pencils $|A|$ and $|A'|$.

We will apply this to $A$ of the form $\ell-\e_i$ and $2\ell-\e_i-\e_j-\e_k-\e_l$, where the exceptional divisors are all distinct. If the blown up points are general enough, the unique members of the linear systems of \eqref{eq:excar} consist  of the union of two smooth rational curves intersecting transversally in two points (on $T$). 
\end{example}

Theorem \ref{thm:deform} also says that limits of line bundles on $X$ are unique up to twisting by multiples of $\xi$. In particular,  there is  an irreducible component $h:\HH \to \DD$ of the relative Hilbert scheme such that $\HH_t:=h^{-1}(t)=|\L|_{\mathfrak{X}_t}|$ for all $t \neq 0$,  and
$\HH_0:= h^{-1}(0)$ has various irreducible components; indeed
\[ \HH_0=\cup_{a \in \ZZ}\Big\{ (C_R,C_P) \; | \; C_R \in |L_R(aT)|, \; C_P \in |L_P(-aT)|, \; C_R \cap T=C_P\cap T \Big\}\]
(cf., e.g., \cite[\S 2]{CD}).  We call the latter the {\it limit linear systems} of $|\L|_{\mathfrak{X}_t}|$. 
We will be interested in limits on $X$ of singular curves on $\mathfrak{X}_t$. For the proof of Theorem \ref{mainthm2} we will need the following general result on limit curve singularities, which is implicitly contained in \cite{galati}.

\begin{lemma}[\cite{galati}]\label{limit-sing-point}
Assume that, for $t\in\mathbb D\setminus \{0\}$ general, there exists a reduced and irreducible curve $C_t\in \HH_t$ 
with a point $p_t$ of multiplicity $m$ as  its only singularity.  Denote by $\mathfrak T\subset \mathfrak H$ the Zariski closure  in $\mathfrak H$  of the family of locally trivial deformations of the $C_t$.   Let $C_0= C_R \cup C_P \subset X$ be a curve parametrized by any point in $\mathfrak T \cap \HH_0$ such that $C_0$ does not contain $T$. Then there is a point $p_0 \in C_0$   (which is the limit of $p_t)$ with one of the following properties:
    \begin{itemize}
  \item[(i)] $p_0 \not \in T$ and $p_0$ is a point of multiplicity $\geq m$ on $C_0$;
  \item[(ii)] $p_0 \in T$ and
    $C_R \cap T=C_P \cap T$ contains $p_0$ with multiplicity $\geq m$.
  \end{itemize}
\end{lemma}

\begin{proof}
  Pick any point in $\mathfrak T \cap \HH_t$ representing a curve $C_0=C_R \cup C_P  \subset X$ not containing $T$. Up to making a base change of a suitable order $k\geq 1$, totally ramified at $0$, we may assume there exists a family of curves rationally defined over 
  $\mathbb D$ with general fibre $C_t$ and special fibre $C_0$ (see, e.g., \cite[pp. 39-43]{GH})\footnote{ To be precise, our assumptions yield
    that there is an irreducible relative Cartier divisor $\mathfrak C\to\mathbb D$ on $\mathfrak X$    such that its general fiber $\mathfrak C_t\subset\mathfrak X_t$  is  algebraically equivalent (on $\mathfrak{X}$)  to $kC_t$ for some $k \in \ZZ^+$ and  contains $C_t$ as an irreducible component with multiplicity $1$, and whose special fiber is $\mathfrak C_0=kC_0$. Equivalently, $k$ is the minimum integer such that there exists a reduced and irreducible  $k$-multisection $\gamma\subset\mathfrak T\subset\mathfrak H$ of $h$ so that $\gamma \cap \HH_0= k[C_0]$ and so that $ \gamma\cap \HH_t$ consists of $k$ reduced points, including the point $[C_t]$. 
    The fact that the general fiber is reduced is a consequence of the smoothness of the general fiber of $h$. In the language of \cite[Def. 1.2]{galati} one says that  $C_0$ is a limit curve of $C_t$ with geometric multiplicity $k$. The needed base change has order $k$.}.
 The  resulting total family $\mathfrak{X}^{\prime}$ will be singular along (the strict transform of) $T$. After $k-1$ blowing-ups we obtain a family of surfaces $\mathcal Y\to \mathbb D$ with smooth total space:
\begin{displaymath}
\xymatrix{\mathcal Y\ar@/_/[ddr]\ar[dr]\ar@/^/[drr]^{f_k} & & \\
&  \mathfrak X^\prime \ar[d]\ar[r] &
\mathfrak X\ar[d]\\
& \mathbb D\ar[r]^{\nu_k}&\mathbb D. }
\end{displaymath}
The special fiber of $\mathcal Y$ is the surface with double normal crossing singularities 
\[ \mathcal Y_0=R\cup \mathcal E_{1}\cup\dots\cup \mathcal E_{k-1}\cup P,\]
consisting of a chain of ruled surfaces $\E_1,\ldots, \E_{k-1}$ over $T$, with $R$ and $P$ attached at its ends such that for each $i \in \{0,\ldots,k-1\}$ the intermediate surface $\E_i$ has two sections $T_i$ and $T_{i+1}$ such that
\[ \E_1 \cap R=T_1=T, \;\; \E_i \cap \E_{i+1}=T_{i+1} \; \mbox{for}\; i \in \{0,\ldots,k-2\}, \;\; \E_{k-1} \cap P=T_k=T, \]
as shown in the following picture:

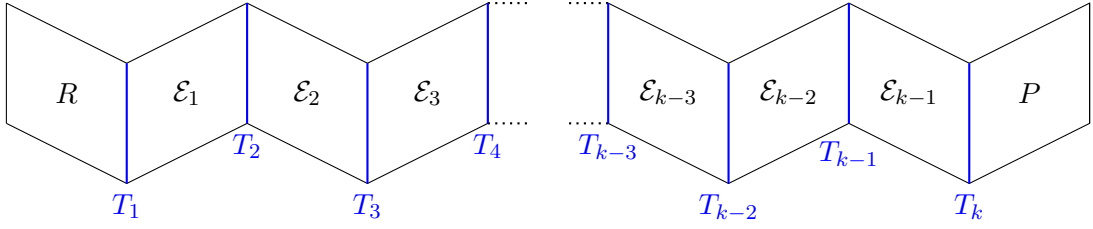
\begin{center}
  \begin{tikzpicture}[scale=0.8]   
    \coordinate (A) at (-10,0);
    \coordinate (A1) at (-9,0.5);
    \coordinate (B) at (-8,-1);
    \coordinate (B1) at (-7,0.5);
    \coordinate (C) at (-6,0);
    \coordinate (C1) at (-5,0.5);
    \coordinate (D) at (-4,-1);
    \coordinate (D1) at (-3,0.5);
		\coordinate (E) at (-2,0);
                \coordinate (E1) at (-1.3,0);
                \coordinate (F1) at (-0.7,0);
                \coordinate (F) at (0,0);
\coordinate (G1) at (1,0.5);
                \coordinate (G) at (2,-1);
\coordinate (H1) at (3,0.5);
\coordinate (H) at (4,0);
\coordinate (I1) at (5,0.5);
\coordinate (I) at (6,-1);
\coordinate (J1) at (7,0.5);
                \coordinate (J) at (8,0);
		\coordinate (K) at (8,2);
		\coordinate (L) at (6,1);
                \coordinate (M) at (4,2);
		\coordinate (N) at (2,1);
                \coordinate (O) at (0,2);
                \coordinate (O1) at (-0.7,2);
                \coordinate (P1) at (-1.3,2);
		\coordinate (P) at (-2,2);
                \coordinate (Q) at (-4,1);
		\coordinate (R) at (-6,2);
                \coordinate (S) at (-8,1);
		\coordinate (T) at (-10,2);
                \draw  (A)--(B)--(C)--(D)--(E)--(P)--(Q)--(R)--(S)--(T);
                \draw[thick,dotted]  (E)--(E1);
                \draw[thick,dotted]  (P)--(P1);
                \draw[thick,dotted]  (F)--(F1);
                \draw[thick,dotted]  (O)--(O1);
                
\draw  (A)--(T);
\draw[thick,blue] (B) node[below] {$T_1$}--(S);
\draw[thick,blue] (C) node[below] {$T_2$}--(R);
\draw[thick,blue] (D) node[below] {$T_3$}--(Q);
\draw[thick,blue] (E) node[below] {$T_4$}--(P);
\draw  (F)--(G)--(H)--(I)--(J)--(K)--(L)--(M)--(N)--(O);
\draw[thick,blue]  (F) node[below] {$T_{k-3}$}--(O);
\draw[thick,blue] (G) node[below] {$T_{k-2}$}--(N);
\draw[thick,blue] (H) node[below=0.1cm] {$T_{k-1}$}--(M);
\draw[thick,blue] (I) node[below] {$T_{k}$}--(L);
\draw (J)--(K);
\draw (A1)  node{$R$};
\draw (B1)  node{$\E_1$};
\draw (C1)  node{$\E_2$};
\draw (D1)  node{$\E_3$};
\draw (G1)  node{$\E_{k-3}$};
\draw (H1)  node{$\E_{k-2}$};
\draw (I1)  node{$\E_{k-1}$};
\draw (J1)  node{$P$};
\end{tikzpicture}
\end{center}
The induced morphism $f_k:\mathcal Y\to\mathfrak X$ contracts $\mathcal E_1\cup \cdots\cup\mathcal E_{k-1}$ onto $T$ and it is a $k:1$ cover of $\mathfrak X$, totally ramified along $R$ and $P$.

We now have an irreducible relative Cartier divisor $\mathfrak C\to\mathbb D$ on $\Y$    such that its general fiber is $\mathfrak C_t=C_t\subset\mathfrak X_t$  and its special fiber is a curve $\mathfrak{C}_0$ such that $f_k(\mathfrak{C}_0)=C_0$, $\mathfrak{C}_0 \cap R=C_R$, $\mathfrak{C}_0 \cap P=C_P$ and $\mathfrak{C}_0 \cap \E_i$ consists of a union of fibers, for every $i=1,..., k-1,$ in such a way that one obtains  chains of fibers in $\mathcal E_{1}\cup\dots\cup \mathcal E_{k-1}$, counted with the right multiplicity, joining the divisor $C_R\cap T_1$ with the divisor $C_P\cap T_{k}$.  

The locus of singular points $p_t$ now form a section $\gamma \subset \mathfrak C$. In particular $p_0^{\prime}:=\gamma \cap \Y_0$ must lie in the smooth locus of $\Y_0$. Being a limit of the $p_t$, which are points of multiplicity $m$ on $C_t$, the point $p_0^{\prime}$ must be a point of multiplicity at least $m$ on $\mathfrak{C}_0$. If $p_0^{\prime}$ lies on $R$ or on $P$, then it is a point of multiplicity $\geq m$ on $C_R$ or on $C_P$, whence $p_0:=f_k(p_0^{\prime})$ lies off $T$ and is 
a point of multiplicity $\geq m$ on $C_0$. This yields case (i).  If
$p_0^{\prime}$ lies on one of the $\E_i$, then $\mathfrak{C}_0 \cap \E_i$ contains a fiber with multiplicity $\geq m$. The corresponding  chain of fibers intersects $T_1$ and $T_k$ at points $q_P$ and $q_R$, respectively, with multiplicity $\geq m$, whence $C_R \cap T_1$ contains $q_R$ and $C_P \cap T_k$ contains $q_P$
both with multiplicity $\geq m$. Hence, $C_R \cap T=C_P \cap T$ contains the point $p_0:=f_k(q_P)=f_k(q_R)=f(k)(p_0^{\prime})$ with multiplicity $\geq m$. This yields case (ii). 
\end{proof}

\begin{remark} \label{rem:piugeneral}
  The  latter lemma holds in a more general setting of a flat family $f:\mathfrak{X} \to \DD$ over the unit disc such that $\mathfrak{X}$ is smooth with special fiber $\mathfrak{X}_0$ a normal crossing intersection of two smooth surfaces and general fiber an arbitrary smooth surface, as long as  some $[C_t]$ belongs to the smooth locus of  the irreducible component $\H$ of the relative Hilbert scheme that we are considering  (see the footnote on the previous page).
\end{remark}

\section{Rational curves on Enriques surfaces} \label{sec:RC}

In this section we will prove Theorem \ref{mainthm1}.

It is well-known that a general Enriques surface does not contain smooth rational curves, equivalently, $(-2)$-curves (cf. references in \cite[p. 577]{cos2}).
In particular, all curves on a general Enriques surface have non-negative self-intersection, whence all effective divisors of positive self-intersection are ample. 

It is also well-known that for an irreducible curve $C$ on an Enriques surface $S$ with $C^2=0$, one either has
\begin{itemize}
\item $\dim |C|=1$ and $|C|$ is an elliptic pencil, containing precisely two nonreduced fibers, $2E$ and $2E'$, with $E' \sim E+K_S$, and a finite number of singular fibers, which on a general $S$ are $12$ irreducible nodal rational curves;
\item $\dim |C|=0$, in which case $|2C|$ is an elliptic pencil as in the previous case, and $C$ is {\it primitive}, that is, indivisible in $\Num S$.
\end{itemize}

To prove Theorem \ref{mainthm1} we will therefore treat the two cases of ample divisors with positive self-intersection and primitive divisors of self-intersection zero.

\begin{prop} \label{prop:1}
Let $(S,L)$ be a general member of a component of the moduli space of polarized Enriques surfaces. If $L$ is not $2$-divisible in $\Num(S)$, then $|L|$ contains no rational curves.
\end{prop}

\begin{proof}
 Let $X$ be as in \S \ref{sec:fl}. By Theorem \ref{thm:deform} there is a flat family $f: \XX \to \DD$ and a line bundle $\L$ on $\XX$ such that
  \begin{itemize}
  \item $S$ is a general fiber of $f$ and $\L|_S=L$;
  \item $f^{-1}(0)=X$.
      \end{itemize}
  Moreover, it is proved in \cite[\S 7]{cdgk} that under the hypothesis that $L$ is not $2$-divisible in $\Num(S)$, one can make sure that $\L|_{R} \cdot \f$ is odd. It follows that  $\L|_{R}(aT) \cdot \f$ is odd for any $a \in \ZZ$. Therefore, all curves in any linear system $|\L|_{R}(aT)|$ dominate $E$ under the projection $\pi:R \to E$. 
  Therefore, they contain some component with positive geometric genus. Thus, the limit linear systems of $|L|$ as $S$ specializes to $\mathfrak{X}_0$ contain no curves whose components are all rational. Hence, $|L|$ cannot contain rational curves. 
  \end{proof}

\begin{prop} \label{prop:2}
Let $S$ be a general Enriques surface. Then any primitive effective isotropic divisor is smooth and irreducible.
\end{prop}

\begin{proof}
  Let $\E_{3,1}$ denote the irreducible component of the moduli space of polarized Enriques surfaces of genus $3$ parametrizing isomorphism classes of pairs $(S,H)$ such that
  $S$ is an Enriques surface and $H \in \Pic S$ is ample with $p_a(H)=3$ (equivalently, $H^2=4$) and $\phi(H)=1$.
  The fact that such pairs form an irreducible component has probably been known to experts for a long time; in the language of Theorem \ref{thm:mainmod} this follows since any such $H$ has a fundamental presentation \eqref{eq:fundpres} of the form $H \sim 2E_1+E_2$ (see, e.g., \cite{cd}, \cite[Lemma 4.18]{cdgk1} or
\cite[Ex. ``Case $\phi_1=1$'', p. 133]{kn-JMPA}). Denoting by $E_1'$ the only member of $|E_1+K_S|$, the only two isotropic divisors computing $\phi(H)$ are $E_1$ and $E'_1$.

  Consider now $X$ as in \S \ref{sec:fl} and the isotropic divisors $E_i^0:=(\f_{\alpha_i},\ell-\e_i)$ for $i \in \{1,2\}$ on $X$, with notation as in Example \ref{ex:iso} above. Then $(E_i^0)':=(\f_{\alpha_i+\eta},\ell-\e_i) \sim E_i^0+K_X$. 
 By Theorems \ref{thm:mainmod} and \ref{thm:deform} the pair 
$(X,2E_1^0+E_2^0)$ deforms, in a one-parameter family, to a general member $[(S,H:=2E_1+E_2)] \in \E_{3,1}$, in such a way that $E^0_1$ and $(E^0_1)'$ deform to $E_1$ and $E'_1$, respectively. Since $E^0_1$ and $(E^0_1)'$ are a transversal union of two smooth curves intersecting only along $T$, the two nodes must smooth (cf., e.g., \cite[Cor. 1]{galati}), so that $E_1$ and $E_1'$ on $S$ are smooth. This means that the locus $\E^{\circ}_{3,1}$ in $\E_{3,1}$ consisting of pairs $(S,H)$ such that both isotropic divisors computing $\phi(H)=1$ are smooth and irreducible is a nonempty open subset of $\E_{3,1}$. Since the natural forgetful morphism  $g:\E_{3,1} \to \E$ is finite, the set $g(\E_{3,1} \setminus \E^{\circ}_{3,1})$ is a proper subset of $\E$. Hence, for a general Enriques surface $S$, all isotropic divisors $E$ computing $\phi(H)=1$ for some ample $H$ of genus $3$ are smooth and irreducible. For any isotropic divisor $E$ we may find nine other isotropic divisors $F$ such that $E \cdot F=1$ (cf. \cite[Cor. 2.5.6]{cd}). Since $S$ is general, the divisor $2E+F$ is ample, and $E \cdot (2E+F)=\phi(2E+F)=1$. Hence, any effective nontrivial isotropic divisor computes $\phi(H)=1$ for some ample $H$ of genus $3$. The result follows.
\end{proof}

\begin{proof}[Proof of Theorem \ref{mainthm1}]
  If $C^2=0$ and $C$ is not $2$-divisible in $\Num(S)$, then $C$ is a
primitive effective isotropic divisor, whence the result follows from Proposition \ref{prop:2}. 

If $C^2>0$, the result follows from Proposition \ref{prop:1} and the fact that the moduli space of polarized Enriques surfaces of a fixed genus has finitely many components. 
\end{proof}

\section{Seshadri constants on Enriques surfaces} \label{sec:ses}

In this section we will prove Theorem \ref{mainthm2}.

We start by making the following:

\begin{defn}
Let $S$ be an Enriques surface and $L$ an effective line bundle or divisor on $S$ with $L^2 \geq 0$. The {\em length of $L$}, denoted by $l(L)$, is the maximal number $l$ of (possibly non-distinct) effective isotropic divisors $E_1,\ldots,E_l$ such that $L \sim E_1+\cdots +E_l$.
\end{defn}

\begin{lemma} \label{lemma:1}
If $L^2>0$, then $\phi(L) \leq l(L)$.
\end{lemma}

\begin{proof}
In the notation of Proposition \ref{prop:exdec}, we have
  \[ \phi(L) \leq E_8 \cdot L =a_1+\cdots+ a_7+a_9+a_{10}+a_0 \leq l(L).\]
\end{proof}

\begin{lemma} \label{lemma:2}
  We have $L^2 \leq l(L)^2+l(L)-2$. If $L^2>0$ and equality occurs, then $\phi(L)=l(L)$.
\end{lemma}

\begin{proof}
  We first prove the inequality $L^2 \leq l(L)^2+l(L)-2$ by induction. The base case is the case $L^2=0$, where it clearly holds.

  If $L^2>0$, pick an isotropic $E$ such that $E \cdot L=\phi(L)$. Since $\phi(L)^2 \leq L^2$ by \cite[Cor. 2.7.1]{cd}, we have $(L-E)^2 = L^2-2\phi(L) \geq \phi(L)^2-2\phi(L)$.
Thus,  $(L-E)^2 \geq 0$. Moreover, one obviously has
$l(L-E) \leq l(L)-1$.
Thus, by induction, we have
  \begin{eqnarray*}
    L^2-2\phi(L) & = & (L-E)^2 \leq l(L-E)^2+l(L-E)-2
    =l(L-E)\left(l(L-E)+1\right)-2
    \\
    & \leq &\left(l(L)-1\right)l(L)-2=l(L)^2-l(L)-2.
  \end{eqnarray*}
  Using the fact that $\phi(L) \leq l(L)$ by Lemma \ref{lemma:1}, we get
  \begin{equation} \label{eq:2.1}
    L^2 \leq l(L)^2-l(L)-2+2\phi(L) \leq l(L)^2-l(L)-2+2l(L) = l(L)^2+l(L)-2,
  \end{equation}
  as desired. 

  Assume now that $L^2>0$ and $L^2 = l(L)^2+l(L)-2$. From \eqref{eq:2.1} we see that we must have $\phi(L)=l(L)$.
\end{proof}

We now give the main ingredient, besides Theorem \ref{mainthm1},  in the proof of Theorem \ref{mainthm2}:

\begin{prop} \label{prop:sessub}
  Let $H$ be a line bundle on an Enriques surface with $H^2>0$. 
  If $C \subset S$ is an irreducible curve with a point of multiplicity $m$ such that $\frac{ C \cdot H}{m} < \phi(H)$, then $C$ is rational and $C^2=m(m-1)-2$.
 Moreover, if $C^2>0$, then $\phi(C)=l(C)=m-1$ and $C$ is of one of the following types
(in the notation of Proposition \ref{prop:exdec}):
  \begin{itemize}
  \item[(i)]  $C \equiv E_9+E_{9,10}$,
  \item[(ii)] $C \equiv 2(E_9+E_{10}+E_{9,10})$,
  \item[(iii)] $C \equiv h(E_9+E_{9,10})+E_{10}$, $\phi(C)=2h+1$,
  \item[(iv)]  $C \equiv (h+1)E_{9,10}+hE_9+E_{10}$, $\phi(C)=2h+2$.
  \end{itemize}
\end{prop}

\begin{proof}
  If $C$ satisfies the given hypotheses, then
 \[ \frac{1}{2}C^2+1=p_a(C) \geq \frac{1}{2}m(m-1)+p_g(C),\]
 whence
 \begin{equation} \label{eq:sessub1}
   C^2 \geq m(m-1)-2+2p_g(C).
 \end{equation}

If $C^2<0$, then we must have $C^2=-2$, $p_g(C)=0$ and $m=1$. 

If $C^2=0$, then $m>1$, by definition of $\phi(H)$ and the assumption that $\frac{C \cdot H}{m} < \phi(H)$. Hence, \eqref{eq:sessub1} implies that $p_g(C)=0$ and $m=2$.  

If $C^2>0$, Lemma \ref{lemma:2} and \eqref{eq:sessub1} yield that $l(C) \geq m-1$.
If $l(C) \geq  m$, then 
\[  \frac{C \cdot H}{m} \geq \frac{l(C)\phi(H)}{m} \geq \frac{ m \phi(H)}{m} \geq \phi(H),\]
a contradiction.
If $l(C) = m-1$, then
$C^2 = m(m-1)-2$, $p_g(C)=0$ and $\phi(C)=l(C)=m-1$. Thus, $C^2=\phi(C)(\phi(C)+1)-2$, whence  \cite[Prop. 1.4]{KL} implies that $C$ is as in one of the cases (i)-(iv).
\end{proof}

We deduce a few consequences. The first is an improvement of \cite[Thm. 3.4]{Sz}:

\begin{corollary} \label{cor:mainthm2_2}
  Let $S$ be an Enriques surface and $H$ a big and nef line bundle on $S$. For any $x \in S$, set $d(H,x):=\min\{R \cdot H\; |\; \mbox{$R$ is a $(-2)$-curve on $S$ passing through $x$}\}$. For all $x \in S$ we have
  \[ \varepsilon(H,x) \geq \min\{d(H,x),\frac{1}{2}\phi(H)\}.\]
\end{corollary}

\begin{proof}
  If $\varepsilon(H,x) \geq  \phi(H)$, there is nothing to prove, so we may assume that $\varepsilon(H,x) < \phi(H)$; in particular,  there exists an irreducible curve $C$ with a point of multiplicity $m$ at $x$ such that $\frac{C \cdot H}{m} < \phi(H)$. 
By Proposition \ref{prop:sessub}, we must have $C^2=m(m-1)-2$.

  If $m=1$, we have $C^2=-2$ and $\frac{C \cdot H}{m}=C \cdot H \geq d(H,x)$. 

  If $m=2$, we have $C^2=0$ and  $\frac{C \cdot H}{m}=\frac{C \cdot H}{2} \geq \frac{\phi(H)}{2}$. 

If $m>2$, then $C^2>0$. By Proposition \ref{prop:sessub}, we have that $l(C)=m-1$. Thus, $\frac{C \cdot H}{m} \geq \frac{l(C)\phi(H)}{m}=\frac{(m-1)\phi(H)}{m} \geq \frac{2}{3}\phi(H)$. This concludes the proof.
  \end{proof}

\begin{corollary} \label{cor:mainthm2_1}
  Let $S$ be an Enriques surface and $H$ a big and nef line bundle on $S$. Except possibly for countably many $x$ or $x$ lying on $(-2)$-curves, we have $\varepsilon(H,x)  \geq \phi(H)$.
\end{corollary}

\begin{proof}
  Assume that $\varepsilon(H,x) < \phi(H)$ for a point $x \in S$. Then, by Proposition \ref{prop:sessub}, there is a rational curve $C$ with a point of multiplicity $m$ at $x$, with $m=1$ occurring only when $C$ is a $(-2)$-curve. Since a rational curve does not move (as a rational curve) in its linear system, there are at most countably many such curves, whence also countably many such $x$ of multiplicity $m \geq 2$. 
\end{proof}

\begin{proof}[Proof of Theorem \ref{mainthm2}]
Clearly $\varepsilon(H) \leq \phi(H)$. Assume that equality does not hold. Then by Proposition \ref{prop:sessub} there exists an irreducible rational curve $C$ with a point $x$ of multiplicity $m$ such that $C^2=m(m-1)-2$ and $\frac{C \cdot H}{m} < \phi(H)$. Since a (very) general Enriques surface does not contain $(-2)$-curves, we must have $C^2 \geq 0$.  

  If $C^2=0$, then $m=2$, and by Proposition \ref{prop:2} the divisor $C$ cannot be primitive. Therefore, $C \cdot H \geq 2\phi(H)$, which implies
  $\frac{C \cdot H}{2} = \phi(H)$, a contradiction.

  Assume now that $C^2>0$. By Proposition \ref{prop:sessub} we have that $m=\phi(C)+1$ and $C$ must be as in one of the cases (i)-(iv). Since $C$ must also be $2$-divisible by Proposition \ref{prop:1}, we see that we must be in case (ii), that is, $C \equiv 2(E_9+E_{10}+E_{9,10})$, with $l(C)=\phi(C)=6$ (again by Proposition \ref{prop:sessub}), whence $m=7$. By Theorem \ref{thm:mainmod} all such pairs $(S,\O_S(C))$ form two irreducible components in the moduli space of polarized Enriques surfaces, determined by whether $\O_S(C)$ is $2$-divisible in $\Pic S$ or not. 

  Consider $X$ as in \S \ref{sec:fl} and isotropic divisors $E_i^0\equiv (\f,\ell-\e_i)$, for $i \in \{1,2\}$, and $E_3^0\equiv (\f_,2\ell-\e_3-\e_4-\e_5-\e_6)$ on $X$, as in Example \ref{ex:iso}. Then $E_1^0\cdot E_2^0=1$ and $E_1^0\cdot E_3^0=E_2^0\cdot E_3^0=2$.
  Set $L^0:=2(E_1^0+E_2^0+E_3^0)\equiv(6\f,8\ell-2(\e_1+\cdots+\e_6))$. 
  By Theorems \ref{thm:mainmod} and \ref{thm:deform}
    the pair 
    $(X,L^0)$ or $(X,L^0+K_X)$ deforms, in a one-parameter family, to $(S,\O_S(C))$. The rational curve $C$ must land in one of the limit linear systems of $|\O_S(C)|$, which means it has a limit of the form $C_R \cup_T C_P$ with $C_R \subset R$ such that $C_R \equiv 6\f(-aT)$ for some $a \in \ZZ$, and $C_P \subset P$ such that $C_P \sim 8\ell-2(\e_1+\cdots+\e_6)+aT$ with $C_R \cap T=C_P \cap T$ (as the limit curve cannot contain $T$, since $T$ is elliptic). Since $C_R$ must contain only rational components, we must have $C_R \cdot \f=0$, whence $a=0$.
    Since any fiber $\f$ intersects $T$ in two distinct points, $C_R \cap T$ has a point of multiplicity at most $6$. Thus, by Lemma \ref{limit-sing-point}, the limit of the $7$-uple point of $C$ cannot land on $T$, but must be a point of multiplicity at least $7$ on $C_R \setminus T$ or $C_P \setminus T$.
Since $C_R$ consists of six  smooth fibers, possibly coinciding, it does not have any point of multiplicity  $>6$. Thus,
$C_P \in |8\ell-2(\e_1+\cdots+\e_6)|$ has a point $x_0$ of multiplicity $m_0 \geq
7$ off $T$.  Since $C_P \cdot (\ell-\e_i)=6<m_0$, and $|\ell-\e_i|$ is a pencil, the member of the pencil passing through $x_0$ must be a component of $C_P$.
Thus, $C_P-\displaystyle\sum_{i=1}^6(\ell-\e_i)=2\ell-(\e_1+\cdots+\e_6)$ is effective, which is impossible, since $y_1,\ldots,y_6$ can be chosen general in $\PP^2$.
This gives the desired contradiction.
\end{proof}


\begin{thebibliography}{[FKP3]}

\bibitem{Bal} J.~Baltes, {\it Singular rational curves on elliptic K3 surfaces},  Math. Nachr. (2023), 1--14, doi: 10.1002/mana.202200228.

\bibitem{ba-sur}
   Th.~Bauer, {\it Seshadri constants on algebraic surfaces},
   Math. Ann. {\bf 313} (1999), 547--583.

\bibitem{prim} Th.~Bauer, S.~Di Rocco, B.~Harbourne, M.~Kapustka, A.~L.~Knutsen, W.~Syzdek, T.~Szemberg, {\it A primer on Seshadri constants}. Interactions of classical and numerical algebraic geometry, 33--70, Contemp. Math., {\bf 496}, Amer. Math. Soc., Providence, RI, 2009, 

\bibitem{CD} C.~Ciliberto, T.~Dedieu, {\it Limits of pluri-tangent planes to quartic surfaces}, Algebraic and Complex Geometry volume 71 of Springer Proc. Math. \& Stat, 2014, 123--199.
 
\bibitem{indam} C.~Ciliberto, T.~Dedieu, C.~Galati, A.~L.~Knutsen, {\it 
A note on Severi varieties of nodal curves on Enriques surfaces}, in "Birational Geometry and Moduli Spaces", Springer INdAM Series {\bf 39}, 29--36 (2018).  

\bibitem{cdgk1} C.~Ciliberto, T.~Dedieu, C.~Galati, A.~L.~Knutsen, {\it  Irreducible unirational and uniruled components of moduli spaces of  polarized Enriques surfaces},
Math. Z.  {\bf 303}:73 (2023), 1--34, doi: 10.1007/s00209-023-03226-5.

\bibitem{cdgk} C.~Ciliberto, T.~Dedieu, C.~Galati, A.~L.~Knutsen, {\it Nonemptiness of Severi varieties on Enriques surfaces}, 
Forum of Mathemtics - Sigma, {\bf 11}:e52 (2023), 1--32, doi:10.1017/fms.2023.47.


 \bibitem{cos2}  F.~R.~Cossec, {\it On the Picard group of Enriques surfaces}, Math. Ann. {\bf 271} (1985), 577--600. 


\bibitem{cd}  F.~R.~Cossec,  I.~V.~Dolgachev,  {\it Enriques surfaces. I.}
  Progress in Mathematics, {\bf 76}. Birkh\"auser Boston, Inc., Boston, MA, 1989.

\bibitem{de}
   J.-P.~Demailly, {\it   Singular Hermitian metrics on positive line bundles},
   Complex algebraic varieties (Bayreuth, 1990), Lect. Notes Math. {\bf 1507},
   Springer-Verlag, 1992, pp. 87--104.



 \bibitem{fri2} R.~Friedman, {\it A new proof of the global Torelli theorem for $K3$ surfaces}, Annals of Math. {\bf 120} (1984), 237--269.


\bibitem{galati} C.~Galati, {\em Degenerating curves and surfaces: first results},
Rend. Circ. Mat. Palermo {\bf 58} (2009), 211--243.



\bibitem{GH} P~.Griffiths, J.~Harris, On the Noether–Lefschetz theorem and some remarks on
codimension-two cycles, Math. Ann. {\bf 271} (1985), 31--51.



  \bibitem{kn-man} A.~L.~Knutsen, {\it On $k$th order embeddings of $K3$ surfaces and Enriques surfaces}, Manuscr. Math. {\bf 104} (2001), 211--237 .


 \bibitem{kn-JMPA} A.~L.~Knutsen, {\it On moduli spaces of polarized Enriques surfaces}, J. Math. Pures Appl. {\bf 144}, 106--136 (2020)


  

\bibitem{KL} A.~L.~Knutsen, A.~F.~Lopez, {\it Brill-Noether theory for curves on Enriques surfaces, I: the positive cone and gonality},  Math. Zeit. {\bf 261}(2009), 659--690. 

 
  \bibitem{KLpn} A.~L.~Knutsen, A.~F.~Lopez, {\it Projective normality and the generation of the
   ideal of an Enriques surface}, Adv. Geom. {\bf 15} (2015), 339--348.


\bibitem{PAG}
   R.~Lazarsfeld, {\it
   Positivity in Algebraic Geometry~I}.
   Springer-Verlag, 2004.



\bibitem{Sz} T.~Szemberg, {\it On positivity of line bundles on {E}nriques surfaces}, Trans. Amer. Math. Soc. {\bf 353} (2001), 4963--4972.


\end{thebibliography}
\end{document}